\documentclass{amsart}

\usepackage{amssymb}
\usepackage{amsthm}
\usepackage{amsmath}
\usepackage{graphicx}
\usepackage[all]{xy}
\usepackage{color}
\usepackage{dsfont}
\usepackage{accents}
\usepackage{color}

\theoremstyle{plain}
\newtheorem{theorem}{Theorem}[section]
\newtheorem{proposition}[theorem]{Proposition}
\newtheorem{lemma}[theorem]{Lemma}

\newtheorem*{conjecture*}{Conjecture}
\theoremstyle{definition}
\newtheorem{definition}[theorem]{Definition}
\newtheorem*{definition*}{Definition}
\newtheorem{example}[theorem]{Example}
\newtheorem*{example*}{Example}
\newtheorem*{notation*}{Notation}
\newtheorem*{notation-conv*}{Notation and convention}
\newtheorem*{convention*}{Convention}
\theoremstyle{remark}
\newtheorem{remark}[theorem]{Remark}
\newtheorem*{remark*}{Remark}

\newcommand{\Z}{\mathbb{Z}}

\newcommand{\Bl}{\text{Bl}}

\begin{document}


\title[]{
  Concordance maps in HFK$^{-}$
}

\author{Lev Tovstopyat-Nelip}

\address{Department of Mathematics, 
  Boston College, 
  Chestnut Hill, Massachussetts 02467}
\email{tovstopy@bc.edu}



\begin{abstract}
  We show that a decorated knot concordance $\mathcal{C}$ from $K_0$ to $K_1$
  induces an $\mathbb{F}[U]$-module homomorphism \[G_{\mathcal{C}}: HFK^{-}(-S^3,K_0) \to HFK^{-}(-S^3,K_1)\] which preserves the Alexander and absolute $\mathbb{Z}_2$-Maslov gradings. Our construction generalizes
  the concordance maps induced on $\widehat{HFK}$
  studied by Juh\'{a}sz and Marengon \cite{JM2}, but uses the description of $HFK^{-}$ as a direct limit of maps between sutured Floer homology groups discovered by Etnyre, Vela-Vick, and Zarev \cite{EVZ}. \end{abstract}


\maketitle

\section{Introduction}
Knot Floer homology is an invariant of knots which categorifies the Alexander polynomial. It was defined independently by Oszv\'{a}th and Szab\'{o} \cite {OS2} and Rasmussen \cite{R}. The simplest version of knot Floer homology is $\widehat{HFK}$, a finitely generated $\mathbb{F}$-vector space\footnote{We use $\mathbb{F}=\mathbb{F}_2$ coefficients throughout this paper.}. This invariant is functorial with respect to decorated knot cobordisms according to Juh\'{a}sz \cite{cob}. The maps between knot Floer groups induced by concordances preserve the Alexander and absolute $\mathbb{Q}$-Maslov gradings. A more powerful flavor of knot Floer homology is $HFK^{-}$, a finitely generated $\mathbb{F}[U]$-module. The goal of this paper is to define maps on $HFK^{-}$ associated to decorated concordances, turning this flavor of knot Floer homology into a functor.

Juh\'{a}sz's maps between knot Floer groups are defined by considering certain cobordisms of sutured manifolds; this approach uses the observation that $\widehat{HFK}$ is naturally isomorphic to the sutured Floer homology of the knot complement with two meridional sutures \cite{sutured}. In \cite{EVZ}, Etnyre, Vela-Vick and Zarev prove that $HFK^{-}$ is isomorphic to a direct limit of maps between sutured Floer homology groups associated to the knot complement. They denote this limit by $\underrightarrow{SFH}$.
Our approach to defining concordance maps on $HFK^{-}$ is to use the sutured cobordism maps defined by Juh{\'a}sz in combination with the sutured Floer-theoretic description of $HFK^{-}$ in \cite{EVZ}. 

Our main result is the following:
\begin{theorem}\label{first}
A decorated concordance $\mathcal{C} = (F,\sigma)$ from $(K_0,P_0)$ to $(K_1,P_1)$ induces an $\mathbb{F}[U]$-module homomorphism
\[G_\mathcal{C}: HFK^{-}(-S^3,K_0)\to HFK^{-}(-S^3,K_1)\] which preserves the Alexander and absolute $\mathbb{Z}_2$-Maslov gradings.
\end{theorem}

 Our maps are natural extensions of those defined in \cite{cob}. If $F_\mathcal{C} : \widehat{HFK}(-S^3,K_0)\to\widehat{HFK}(-S^3,K_1)$ is the map induced by a decorated concordance $\mathcal{C}$, then we have the following commutative diagram (see Proposition \ref{behave} and the following Remark):

\[
\xymatrix{
HFK^{-}(-S^3,K_0) \ar[d]^{p_*} \ar[r]^{G_\mathcal{C}} &HFK^{-}(-S^3,K_1)\ar[d]^{p_*}\\
\widehat{HFK}(-S^3,K_0) \ar[r]^{F_\mathcal{C}} &\widehat{HFK}(-S^3,K_1)}
\]
where the maps $p_{*}$ are induced by settings $U=0$ on the chain level.

While putting the final touches on this article, Zemke \cite{Z} posted a paper containing another approach to defining link cobordism maps on $HFK^{-}$. See also Alishahi-Eftekhary in recent days \cite{AE}. Our approach is quite different from those two. Among other things, it is more contact-geometric in nature. One of our hopes is that it may be better suited to understanding the relationship between Lagrangian concordances and knot Floer homology. More precisely, we conjecture that the Legendrian invariant in $HFK^{-}$ defined in \cite{LOSS} should behave functorially with respect to the map induced by a Lagrangian concordance between two Legendrian knots; see Baldwin-Sivek \cite{BS} for a version of this result in the monopole Floer analogue of $\widehat{HFK}$.

We outline the structure of the paper. In Section \ref{sec:cobmapssfh} we review the $SFH$ TQFT as in \cite{cob}. In Section \ref{sec:concmapshfk} we review the maps induced by decorated concordances on $\widehat{HFK}$ as defined in \cite{JM1}. In Section \ref{sec:sfhlim} we review the construction of $\underrightarrow{SFH}$. In section five we finally define the maps induced by decorated concordances on $HFK^{-}$ and prove some commutative diagrams involving maps to $\widehat{HFK}$ and $\widehat{HF}$. In Section \ref{sec:gradings} we show that our maps preserve the Alexander and the absolute Maslov $\mathbb{Z}_2$-gradings.


\section*{Acknowledgments}
The author is grateful to his advisor, John Baldwin, for many helpful conversations. The author also thanks John Etnyre for an interesting discussion.

\section{Cobordism maps in sutured Floer homology}
\label{sec:cobmapssfh}

We begin by introducing sutured manifolds and cobordisms between them, and then review maps induced by cobordisms on sutured Floer homology. Consult \cite{cob} for a complete treatment. 

\subsection{Cobordisms of sutured manifolds}
\begin{definition}[{\cite[Definition~2.6]{Gabai}}]

A \emph{sutured manifold} $(M,\Gamma)$ is a compact oriented $3$-manifold $M$ with nonempty boundary, together with an oriented compact subsurface $\Gamma\subset \partial M$ which is the union of disjoint annuli. The oriented core of each annulus is called a \emph{suture}. Each component of $R(\Gamma) = \partial M \setminus \Gamma$ is oriented so that $\partial R(\Gamma)$ is compatible with the sutures. $R_+(\Gamma)$ (or $R_- (\Gamma)$) denotes the components of $R(\Gamma)$ whose normal vectors point out of (or into) $M$.
\end{definition}

\begin{definition}
A sutured manifold $(M,\Gamma)$ is \emph{balanced} if $M$ has no closed components,
$\chi(R_+(\Gamma))=\chi(R_-(\Gamma))$, and each boundary component of $M$ has atleast one suture.
\end{definition}

We view~$\Gamma$ as a ``thickened'' oriented 1-manifold. We will only be interested in connected sutured manifolds having connected boundary.

\begin{definition}
\label{def:equivalent}
Let $\xi_0$ and $\xi_1$ be contact structures on $(M,\Gamma)$ such that $\partial M$ is convex with dividing set $\Gamma$ with respect to both contact structures.
We say that $\xi_0$ and $\xi_1$ are \emph{equivalent} if there is a family of contact structures $\{\xi_t : t\in I\}$ interpolating between $\xi_0$ and $\xi_1$ such that $\partial M$ is convex with dividing set $\Gamma$ with respect to each contact structure in the family. We let $\xi_0\sim \xi_1$ denote equivalence, and let $[\xi]$ denote the equivalence class of $\xi$.
\end{definition}

\begin{definition}
A \emph{cobordism of sutured manifolds} from $(M_0,\Gamma_0)$ to $(M_1, \Gamma_1)$ consists of a triple $\mathcal{W}=(W, Z, [\xi])$, where

\begin{itemize}
  \setlength{\itemsep}{1pt}
  \setlength{\parskip}{0pt}
  \setlength{\parsep}{0pt}
\item{$W$ is a oriented compact $4$-manifold with nonempty boundary,}
\item{$Z \subseteq \partial W$ is a compact $3$-manifold with nonempty boundary, whose complement in $\partial W$ consists of the disjoint union $-\accentset{\circ}{M_0} \sqcup \accentset{\circ}{M_1}$.
$Z$ is a sutured manifold with sutures $\Gamma_0 \cup \Gamma_1$,}
\item{$\xi$ is a positive contact structure on $Z$ such that $\partial Z$ is convex with dividing set $\Gamma_i$ on $\partial M_i$ for $i \in \{0,1\}$.}
\end{itemize}
\end{definition}

\begin{definition}
Let $\mathcal{W}=(W, Z, [\xi])$ and $\mathcal{W'}=(W', Z', [\xi'])$ be sutured cobordisms from $(M_0, \Gamma_0)$
to $(M_1, \Gamma_1)$. $\mathcal{W}$ and $\mathcal{W}'$ are called \emph{equivalent} if there is an orientation preserving diffeomorphism $d: W\to W'$ which carries $(Z,\xi)$ to $(Z',\xi')$ and restricts to the identity on $M_0 \sqcup M_1$.
\end{definition}

\begin{definition}
A cobordism $\mathcal{W} = (W, Z, [\xi])$ between balanced sutured manifolds $(M_0, \Gamma_0)$ and $(N, \Gamma_1)$
is called a \emph{boundary cobordism} if there exists a deformation retraction $r$ of $W$ to $M_0\cup (-Z)$ such that $r_1|_N : N\to M_0\cup (-Z)$
is an orientation preserving diffeomorphism.
\end{definition}

\begin{definition}
A cobordism $\mathcal{W} = (W, Z, [\xi])$ between balanced sutured manifolds $(M_0, \Gamma_0)$ and $(M_1, \Gamma_1)$
is called \emph{special} if
\begin{enumerate}
\item{$Z = \partial M_0 \times I$ is the trivial cobordism from $\partial M_0$ to $\partial M_1$;}
\item{$\xi$ is an $I$-invariant contact structure on $Z$, each $\partial M_0 \times \left\{t\right\}$ is
convex with respect to the contact vector field $\partial/\partial t$ and has dividing set $\Gamma_0 \times \left\{t\right\}$.} In particular, $\Gamma_0 = \Gamma_1$.
\end{enumerate}
\end{definition}

\begin{remark}\label{remark1}
It is noted in \cite{cob} that every sutured cobordism can be seen as the composition of a boundary
cobordism and a special cobordism as follows, cf.~\cite[Definition~10.1]{cob}.
Let $\mathcal{W} = (W, Z, [\xi])$ be a cobordism of balanced sutured manifolds from $(M_0, \Gamma_0)$ to
$(M_1, \Gamma_1)$.

Let $(N,\Gamma_1)$ be the sutured manifold $(M_0 \cup (-Z), \Gamma_1)$.
Then we can think of the cobordism $\mathcal{W}$ as a composition $\mathcal{W}^s \circ \mathcal{W}^b$,
where $\mathcal{W}^b$ is a boundary cobordism from $(M_0, \Gamma_0)$ to $(N, \Gamma_1)$
and $\mathcal{W}^s$ is a special cobordism from $(N,\Gamma_1)$ to $(M_1, \Gamma_1)$.
\end{remark}

\begin{remark}\label{remark2}
Moreover, every sutured cobordism can be seen as the composition of first a special cobordism and then a boundary cobordism. A special cobordism can be thought of as a trace cobordism of a sequence of handle attachments whose attaching regions lie in the interior of the sutured manifold. The cobordism $\mathcal{W}^s$ above is obtained from attaching handles to $(N,\Gamma_1)$. Noting that the interior of $(N,\Gamma_1)$ is the same as that of $(M_0,\Gamma_0)$, attaching these same handles to $(M_0,\Gamma_0)$ yields a special cobordism $\widetilde{\mathcal{W}^s}$ to $(\widetilde{N},\Gamma_0)$. Now we may think of $\mathcal{W}$ as a composition $\widetilde{\mathcal{W}^b}\circ\widetilde{\mathcal{W}^s}$, where $\widetilde{\mathcal{W}^b}$ is a boundary cobordism from 
$(\widetilde{N},\Gamma_0)$ to $(M_1,\Gamma_1)$. Also note that $\overline{(M_1,\Gamma_1)\setminus(\widetilde{N},\Gamma_0)} \approx \overline{(N,\Gamma_1)\setminus(M_0,\Gamma_0)}\approx -Z$.

\end{remark}

\begin{remark}(\cite[Remark~2.13]{cob})\label{remark3}
Given a cobordism $\mathcal{W} = (W,Z,[\xi])$ from $(M_0, \Gamma_0)$ to
$(M_1, \Gamma_1)$, consider $\overline{\mathcal{W}} = (W,Z,[-\xi])$. Viewing
$\partial W$ as $-(-M_1)\cup Z \cup (-M_0)$ we see that $\overline{\mathcal{W}}$
is a cobordism from $(-M_1,-\Gamma_1)$ to $(-M_0,-\Gamma_0)$. We refer to
$\overline{\mathcal{W}}$ as $\mathcal{W}$ ``turned upside down".

\end{remark}

\subsection{Induced maps on sutured Floer homology}

In \cite{sutured}, Juh\'{a}sz defines $SFH(M,\Gamma)$, the \emph{sutured Floer homology} of a balanced sutured manifold $(M,\Gamma)$. $SFH(M,\Gamma)$ is an $\mathbb{F}$-vector space which splits over the relative Spin$^c$ structures on $(M,\Gamma)$:
\[
SFH(M,\Gamma) = \bigoplus _{\mathfrak{s} \in \  Spin^c (M,\Gamma)} SFH(M,\Gamma,s).
\]
Sutured Floer homology generalizes both $\widehat{HF}$ defined in \cite{OS1} and $\widehat{HFK}$ defined in \cite{OS2} and \cite{R}.

The next two examples are \textbf{Theorem 2.10} and \textbf{Theorem 2.11} from \cite{cob}.
\begin{example}
Let $Y$ be a closed, oriented, 3 manifold. Let $B^3\subset Y$ be an open neighborhood of some point $p\in Y$ and $\Gamma$ consist of a single suture on the boundary of $Y\setminus B^3$; then $SFH(Y\setminus B^3,\Gamma)$ is canonically isomorphic to $\widehat{HF}(Y)$.  \\
\end{example}

A relative $Spin^{\mathbb{C}}$ structure on $(Y,\Gamma)$, is a homology class of vector fields which obey certain boundary conditions, namely the vector fields must point out of $Y$ on all of $R_+{\Gamma}$, into $Y$ on all of $R_-{\Gamma}$, and along $\Gamma$ they flow in the positive $\Gamma$ direction. The space of such vector fields on $\partial Y$ is contractible, so we may fix a vector field $v_0$ on the boundary. If $\partial Y \approx T^2$ and $\Gamma$ consists of two parallel curves, we may choose $v_0$ so that $v_0 ^\perp$ (fix an auxiliary metric) induces a standard characteristic foliation on the torus (see section 2.1 of \cite{EVZ} for the notion of such a foliation). In this case, $v_0^\perp$ is a trivial bundle, so it admits a section $t_0$. Define $c_1(s,t_0)$, the relative Chern class of $s$ with respect to $t_0$, as the obstruction of extending $t_0$ to a nonzero section of $v^\perp$ where $v$ is a vector field representing $s$. There is a natural map from the chain complex of sutured Floer homology to relative $Spin^\mathbb{C}$ structures. See section 2.5 of \cite{EVZ} for a nice review of relative $Spin^\mathbb{C}$ structures.

\begin{example}\label{EX}

Let $K\subset Y$ be a null-homologous smooth knot. Let $Y(K)$ denote the the complement of a neighborhood of $K$ in $Y$ and $\Gamma_\mu$ denote a pair of oppositely oriented meridonal sutures on $\partial Y(K)$; then $SFH(Y(K),\Gamma_\mu)$ is canonically isomorphic to $\widehat{HFK}(Y,K)$. The Alexander grading of a generator $x$ (of the sutured Floer chain complex) may be written as follows
\[
A_{[S,\partial S]}(x) = \frac{1}{2}\langle c_1 (s(x),t_\mu), [S,\partial S]\rangle
\]
where $S$ is a Seifert surface for the knot $K$, $s(x)\in Spin^\mathbb{C}(Y(K),\Gamma_\mu)$ is the relative $Spin^\mathbb{C}$ structure associated to the generator, and $t_\mu$ is a nice vector field along the boundary which has the desired properties with respect to $\Gamma_\mu$. Whenever $Y$ is a $\mathbb{Q}HS^3$, all Seifert surfaces are homologous in the knot exterior, so in this setting the grading is independent of the choice of a Seifert surface.

\end{example}

\begin{theorem}[{\cite[Theorem~11.12]{cob}}]\label{tqft}
Sutured Floer homology defines a functor from the category of balanced sutured manifolds and equivalence classes of cobordisms between them to the category \textbf{Vect}$_{\mathbb{F}}$, which is a $(3+1)$-dimensional
TQFT in the sense of \emph{\cite{Atiyah}} and \emph{\cite{Blanchet}}. 
\end{theorem}

Following \cite{cob}, we outline the construction of the map on sutured Floer homology induced by a cobordism of balanced sutured manifolds.
Let $\mathcal{W} = (W,Z,[\xi])$ be a cobordism of balanced sutured manifolds between $(M_0,\Gamma_0)$ and $(M_1,\Gamma_1)$ such that $Z$ is connected. We may view the cobordism as a composition $\mathcal{W} = \mathcal{W}^s\circ \mathcal{W}^b$ (see Remark \ref{remark1}). $\mathcal{W}^b$ is a boundary cobordism from $(M_0,\Gamma_0)$ to $(N,\Gamma_1)$, and $\mathcal{W}^s$ is a special cobordism from $(N,\Gamma_1)$ to $(M_1,\Gamma_1)$. Since $(M_0,\Gamma_0)$ is a sutured submanifold of $(N,\Gamma_1)$ in the sense of \cite{HKM}, hence there is a contact gluing map \[\phi_{-\xi} : SFH(M_0,\Gamma_0)\to SFH(N,\Gamma_1).\] The map associated to $\mathcal{W}^b$, $\mathcal{F}_{\mathcal{W}^b}$, is defined to be the contact gluing map $\phi_{-\xi}$.

The special cobordism $\mathcal{W}^s$ can be written as $\mathcal{W}_3\circ\mathcal{W}_2\circ\mathcal{W}_1$ where $\mathcal{W}_i$ is the trace of a collection of index i handle attachments. The map $\mathcal{F}_{\mathcal{W}_i}$ associated to a trace cobordism of sutured manifolds is defined analogously to how maps associated to cobordisms are defined for Heegaard Floer homology in \cite{OS1}. The map associated to a trace cobordism is defined as $\mathcal{F}_{\mathcal{W}^s} = \mathcal{F}_{\mathcal{W}_3}\circ\mathcal{F}_{\mathcal{W}_2}\circ \mathcal{F}_{\mathcal{W}_1}$. Finally, we define the map associated to $\mathcal{W}$, $\mathcal{F}_\mathcal{W}$, as the composition $\mathcal{F}_{\mathcal{W}^s}\circ \mathcal{F}_{\mathcal{W}^b}$. Each map defined above admits refinements over relative Spin$^c$ structures.

\begin{remark}\label{order}
Recall that in Remark \ref{remark2} we may write $\mathcal{W} = \widetilde{\mathcal{W}^b}\circ\widetilde{\mathcal{W}^s}$. By Theorem \ref{tqft} we have that $\mathcal{F}_{\mathcal{W}^s}\circ \mathcal{F}_{\mathcal{W}^b}
=\mathcal{F}_\mathcal{W} =\mathcal{F}_{\widetilde{\mathcal{W}^b}}\circ\mathcal{F}_{\widetilde{\mathcal{W}^s}}$.
\end{remark}

\section{Concordance maps in $\widehat{HFK}$}
\label{sec:concmapshfk}

Below we describe Juh\'{a}sz's construction of concordance maps on $\widehat{HFK}$, originally developed in \cite{cob} and studied further in \cite{JM2}.

\begin{definition}
Let $L\subset M$ be a properly embedded submanifold of a smooth manifold. Fix an auxiliary metric on $M$. For each $p\in L$, let $UN_p L$ denote the fiber of the unit normal bundle of $L$ at $p$. Modifying $M$ by replacing each point of $p\in L$ with $UN_p L$ we obtain the \emph{spherical blowup} of $M$ along $L$, denoted BL$_L (M)$. $M$ is diffeomorphic to $M\setminus nbhd(L)$.
For more details consult
Arone and Kankaanrinta~\cite{AK}.
\end{definition}

\begin{definition}
A \emph{decorated knot} is a pair $(K,P)$, where $K\subset S^3$ is a knot, and $P \subset K$ is a pair of points. 
$P$ splits $K$ into two arc components, $R_+ (P)$ and $R_- (P)$; this is part of the data of a decorated knot.

We canonically assign a balanced sutured manifold $S^3(K,P) = (M,\Gamma)$ to every decorated knot $(K,P)$.
Let $M = \Bl_K(S^3)$ and $\Gamma = \bigcup_{p \in P} UN_pK$. Furthermore,
\[
R_\pm(\Gamma) := \bigcup_{x \in R_\pm(P)} UN_xK,
\]
oriented as $\pm \partial M$, and we orient~$\Gamma$ as~$\partial R_+(\Gamma)$.
In other words, $M\approx S^3 (K)$, and $\Gamma$ consists of two oppositely oriented meridional sutures. The blowup perspective will be useful when defining the sutured cobordism associated to a decorated concordance.
\end{definition}

\begin{definition}
We say that the pair $\mathcal{C} = (F,\sigma)$ is a \emph{decorated knot concordance} from $(K_0,P_0)$ to $(K_1,P_1)$ if
\begin{enumerate}
\item $F$ is a concordance (in particular, an annulus) from $K_0$ to $K_1$,
\item $\sigma$ consists of two properly embedded arcs connecting the two components of $\partial F$. One arc goes from $R_+ (P_0)$ to $R_+ (P_1)$ (or $R_-(P_1)$), and the other goes from $R_-(P_0)$ to $R_-(P_1)$ (or $R_+(P_1)$).
\end{enumerate}

In \cite{JM2}, Juh\'{a}sz and Marengon consider annular cobordisms of knots in integer homology $S^3\times I$'s. We restrict to concordances in $S^3\times I$.
\end{definition}

\begin{definition}
To $(F,\sigma)$, a decorated knot concordance from $(K_0,P_0)$ to~$(K_1,P_1)$, we associate a cobordism $\mathcal{W} = \mathcal{W}(F,\sigma)$ as follows.

Arbitrarily split $F$ into $R_+(\sigma)$ and~$R_-(\sigma)$ so that $R_+(\sigma) \cap R_-(\sigma) = \sigma$,
and orient $F$ so that $\partial R_+(\sigma)$ 
crosses $P_0$ from $R_+(P_0)$ to $R_-(P_0)$ and $P_1$ from~$R_-(P_1)$ to $R_+(P_1)$.
$\mathcal{W} = (W,Z,[\xi])$, where $W = \Bl_F(S^3 \times I)$ and $Z = UNF \approx T^2\times I$, oriented
as a submanifold of~$\partial W$, finally $\xi = \xi_{\sigma}$ is the unique $S^1$-invariant contact structure
with dividing set~$\sigma$ on~$F$ and convex boundary~$\partial Z$
with dividing set projecting to~$P_0 \cup P_1$. See \cite{L} for the construction of $\xi_\sigma$.
\end{definition}

In fact $\xi$ can be ``straightened out" so that it is not only $S^1$ invariant, but also $I$ invariant. In this way, we can view $\mathcal{W}$ as a special cobordism of sutured manifolds from $(S^3(K_0),\Gamma_\mu)$ to $(S^3(K_1),\Gamma_\mu)$. See \cite[Section~5.3]{JM2} for details.

\begin{remark}
The sutured cobordism $\mathcal{W}$ above induces a map $SFH(S^3(K_0),\Gamma_\mu)\to SFH(S^3(K_1),\Gamma_\mu)$. Recall that there is a canonical isomorphism $SFH(S^3(K),\Gamma_\mu) \approx \widehat{HFK}(S^3,K)$, so there is an induced map $\widehat{HFK}(S^3,K_0)\to \widehat{HFK}(S^3,K_1)$. This is precisely the map constructed by Juh\'{a}sz in \cite{cob}. 
\end{remark}

\begin{remark}\label{orientation}
To construct maps associated to a decorated concordance in $HFK^{-}$ we are forced to use different orientation conventions (this is a result of the contact geometric nature of the construction of $\underrightarrow{SFH}$).
Let $\mathcal{C} = (F,\sigma)$ be a decorated concordance from $(K_0,P_0)$ to $(K_1,P_1)$. Let $r:S^3\times I \to S^3\times I$ be the diffeomorphism defined by reflecting the second coordinate about $1/2$. $(r(F),r(\sigma))$ is a decorated concordance from $(K_1,P_1)$ to $(K_0,P_0)$. Let $\mathcal{W}'$ be the associated cobordism of sutured manifolds from $(S^3(K_1),\Gamma_\mu)$ to $(S^3(K_0),\Gamma_\mu)$. Turning the cobordism $\mathcal{W}'$ ``upside down" (see Remark \ref{remark3}) we obtain $\mathcal{W}$, a sutured cobordism from $(-S^3(K_0),-\Gamma_\mu)$ to $(-S^3(K_1),-\Gamma_\mu)$. $\mathcal{W}$ induces a map 
\[\mathcal{F}_\mathcal{W} : SFH(-S^3(K_0),-\Gamma_\mu)\to SFH(-S^3(K_1),-\Gamma_\mu).\] Let 
\[F_\mathcal{C}:\widehat{HFK}(-S^3,K_0)\to \widehat{HFK}(-S^3,K_1)\]
denote the map induced by $\mathcal{F}_\mathcal{W}$ and the identification of $\widehat{HFK}(-S^3,K_i)$ with $SFH(-S^3(K_i),-\Gamma_\mu)$.
\end{remark}

\section{The sutured limit homology package}
\label{sec:sfhlim}

In this section, we provide a review of Etnyre, Vela-Vick, and Zarev's \cite{EVZ} limit construction of $HFK^{-}$. We only consider knots in $S^3$ (this is all we need), although everything in this section holds for an arbitrary null-homologous knot in a closed, orientable $Y^3$. 

Let $K\subset S^3$ be a knot. Let $\lambda$ denote the preferred longitude of $K$ and $\mu$ denote the meridian as before. Let $\Gamma_i$ denote a pair of oppositely oriented $(\lambda - i \mu)$ sutures on $\partial S^3 (K)$. We can think of $(S^3(K),\Gamma_i)$ as a sutured submanifold of $(S^3(K),\Gamma_{i+1})$ where $B_i = \overline{(S^3(K),\Gamma_{i+1})\setminus(S^3(K),\Gamma_i)} \approx T^2 \times [0,1]$. Here we are identifying the boundary of $(S^3(K),\Gamma_i)$ with $T^2 \times \{1\}$. Up to fixing the characteristic foliations on the boundary, $B_i$ admits two contact structures $\xi_+$ and $\xi_-$ such that the boundary is convex having dividing curves which agree with the sutures. 

Honda classifies the minimally twisting tight contact structures on $T^2 \times I$ in \cite{Honda}. Honda shows that the contact structures above may be obtained by taking an $I$-invariant contact structure on $T^2 \times [0,1]$ having diving set $\Gamma_i \times \{t\}$ on $T^2\times \{t\}$, and then attaching a positive or negative (see figure, along $\gamma_+$ for positive and $\gamma_-$ for negative), for $\xi_+$ and $\xi_-$ respectively, bypass on the convex surface $T^2\times \{0\}$. Note that this bypass is attached to $-(T^2 \times \{0\})\subset \partial (T^2 \times I)$, i.e. on the negative side of the surface $T^2 \times \{0\}$.

\begin{center}\begin{figure}[h!]
	\centering
	\def\svgwidth{250pt}
	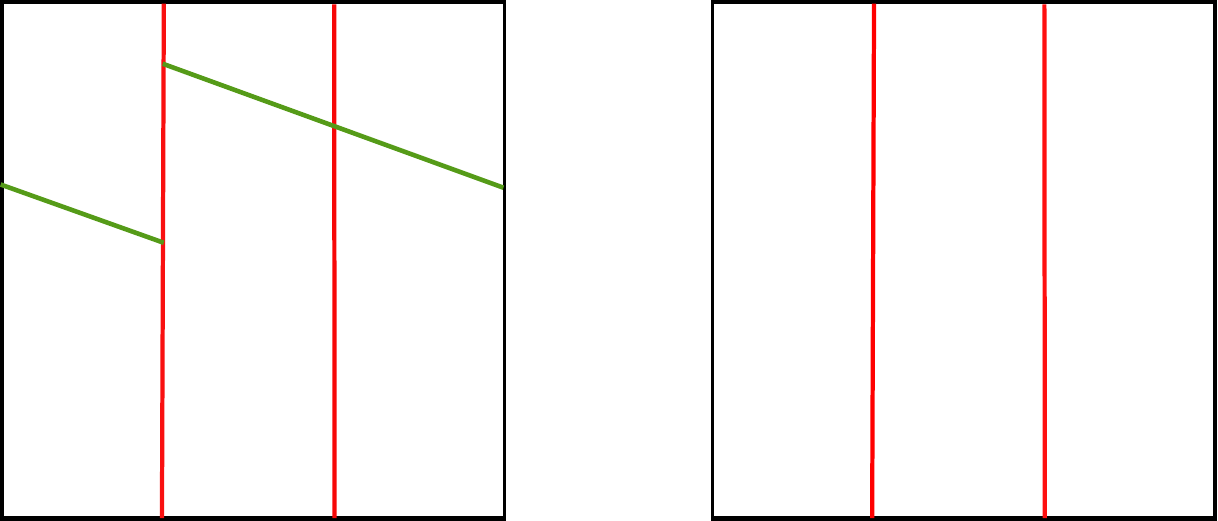
	
\end{figure}\end{center}

Let
\[
\phi_{+,-} : SFH(-S^3 (K),-\Gamma_i) \to SFH(-S^3(K),-\Gamma_{i+1})
\]
denote the contact gluing maps defined in \cite{HKM} induced by $\xi_{+}$ and $\xi_-$ respectively. Taking the direct limit of the above groups with respect to the maps $\phi_-$ gives us the \emph{sutured limit homology} of $K$:
\[
\underrightarrow{SFH} (-S^3,K) := \varinjlim SFH(-S^3(K),-\Gamma_i).
\]

The well-definedness of the contact gluing map implies that the maps $\phi_-$ and $\phi_+$ commute. Thus the latter maps induce a well-defined endomorphism $\Psi$ of $\underrightarrow{SFH}(-S^3,K)$. Endowing $\underrightarrow{SFH}(-S^3,K)$ the structure of an $\mathbb{F}[U]$-module, where the U-action is given by $\Psi$, the following has been shown:

\begin{theorem} \cite[Theorem 1.1]{EVZ}
Let $K\subset S^3$ be a smooth knot. There exists an isomorphism of graded $\mathbb{F}[U]$-modules
\[
I_- :\underrightarrow{SFH}(-S^3,K)\xrightarrow{\approx} HFK^{-} (-S^3,K).
\]
\end{theorem}

\begin{remark}\label{alex}(see \cite[Section 3.3 and Propositions 12.2-12.2]{EVZ})
An Alexander grading may also be defined for the sutured limit homology of $K$. The maps $\phi_+,\phi_-$ used in defining $\underrightarrow{SFH}(-S^3,K)$ are both Alexander homogeneous of degree $-1/2$. To get a well defined Alexander grading on $\underrightarrow{SFH}(-S^3,K)$, one must introduce a grading shift for the complexes at each level involved in computing the direct limit. Part of Theorem 4.1 tells us that this Alexander grading agrees with the usual Alexander grading on $HFK^{-}(-S^3,K)$.\\
\end{remark}

\begin{remark}\label{maslov}
Let $K\subset S^3$ be a nullhomologous knot. An absolute $\mathbb{Z}_2$-Maslov grading may be defined for $\underrightarrow{SFH}(-S^3,K)$ which agrees with the usual Maslov grading for $HFK^{-}(-S^3,K)$. The maps $\phi_-$ and $\phi_+$ are Maslov homogeneous of degree $0$, so the usual Maslov grading on $SFH(-S^3(K),-\Gamma_0)$ induces a grading on the direct limit \cite[Section 12.2]{EVZ}. 

\end{remark}

Each sutured manifold $(S^3(K), \Gamma _i)$ can be viewed as a submanifold of $(S^3 (K),\Gamma_\mu)$ so that $\overline{(S^3 (K),\Gamma_\mu)\setminus(S^3(K), \Gamma _i)} \approx T^2\times [0,1]$. There are again two tight contact structures $\xi_+$ and $\xi_-$ on $T^2 \times [0,1]$ with convex boundary having dividing sets $\Gamma_\mu \cup \Gamma_i$. Let 
\[
\phi_{SV}:SFH(-S^3(K),-\Gamma_i)\to SFH(-S^3(K),-\Gamma_\mu)
\]
be the contact gluing map corresponding to $\xi_-$. The classification of tight contact structures on thickened tori, along with well-definedness of the contact gluing map, tells us that the maps $\{\phi_{SV}\}$ induce a well-defined map $\Phi_{SV}$ on the direct limit.

\begin{theorem} \cite[Theorem 1.3]{EVZ}\label{SVC}
Let $K\subset S^3$ be a smooth knot. The following diagram commutes:
\[
\xymatrix{
\underrightarrow{SFH}(-S^3,K) \ar[d]^{\Phi_{SV}} \ar[r]^{I_-} &HFK^{-}(-S^3,K)\ar[d]^{p_*}\\
SFH(-S^3(K),-\Gamma_\mu) \ar[r]^{\approx} &\widehat{HFK}(-S^3,K)}
\]
where $p_*$ is the map induced by setting $U=0$ on the chain level. The bottom isomorphism is the canonical one.
\end{theorem}

\begin{remark}\label{SV}
The contact structure $\xi_-$ used in defining the Stipsicz-V\'{e}rtesi map above is also induced by a bypass attachment. We use this later to show that the concordance maps on $HFK^-$ generalize the concordance maps on $\widehat{HFK}$.
\end{remark}

We may also attach a contact 2-handle to each $(S^3(K),\Gamma_i)$. Let \[\phi_C: SFH(-S^3(K),-\Gamma_i)\to SFH(-(S^3\setminus B^3),\Gamma)\] (where $\Gamma$ is a single suture on the $S^2$ boundary) denote the induced gluing map. The maps $\phi_C$ again induce a well defined map, $\Phi_C$ on the direct limit.

\begin{theorem} \cite[Theorem 1.4]{EVZ}\label{C2H}
Let $K\subset S^3$ be a smooth knot. The following diagram commutes:
\[
\xymatrix{
\underrightarrow{SFH}(-S^3,K) \ar[d]^{\Phi_{C}} \ar[r]^{I_-} &HFK^{-}(-S^3,K)\ar[d]^{\pi_*}\\
SFH(-(S^3\setminus B^3),-\Gamma) \ar[r]^{\approx} &\widehat{HF}(-S^3)}
\]
where $\pi_*$ is the map induced by setting $U=1$ on the chain level. The bottom isomorphism is the canonical one.
\end{theorem}

\section{Concordance maps in $HFK^{-}$}
\label{sec:concmapshfk-}

We now turn to the proof of our main result, Theorem \ref{first}.
\begin{proof}
We define a map $\mathcal{G}_\mathcal{W} :\underrightarrow{SFH}(-S^3,K_0)\to\underrightarrow{SFH}(-S^3,K_1)$ induced by $\mathcal{C}$ in Proposition \ref{main}. The map $G_\mathcal{C}$ is induced by $\mathcal{G}_\mathcal{W}$ and the identification of $\underrightarrow{SFH}(-S^3,K_i)$ with $HFK^{-}(-S^3,K_i)$. We show $G_\mathcal{C}$ preserves gradings in Propositions \ref{palex} and \ref{pmaslov}.
\end{proof}

\begin{proposition}\label{main}
A decorated concordance $\mathcal{C} = (F,\sigma)$ from $(K_0,P_0)$ to $(K_1,P_1)$ induces an $\mathbb{F}[U]$-module homomorphism \[\mathcal{G}_\mathcal{W} :\underrightarrow{SFH}(-S^3,K_0)\to\underrightarrow{SFH}(-S^3,K_1).\]
\end{proposition}
\begin{proof}
Let $(F,\sigma)$ be a decorated concordance from $(K_0,P_0)$ to $(K_1,P_1)$. Let $\mathcal{W} = (W,Z,[\xi])$ be the cobordism of sutured manifolds considered in Remark \ref{orientation} from $(-S^3(K_0),-\Gamma_\mu)$ to $(-S^3(K_1),-\Gamma_\mu)$. We identify $Z$ with $T^2 \times I$, under this identification $\xi$ is an $I$-invariant contact structure, and each $T^2\times \{t\}$ is convex with dividing set two oppositely oriented meridional curves. Consider the bypass attachment curve $\gamma$ shown below

\begin{center}\begin{figure}[h!]
	\centering
	\def\svgwidth{110pt}
	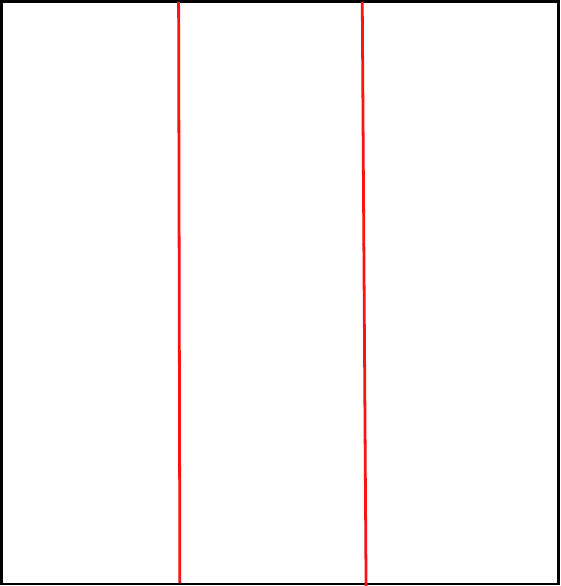
	
\end{figure}\end{center}

Attaching a thickened bypass along $\gamma\times I$ we obtain a new $I$ invariant contact structure $\xi _1$ on $T^2\times I$, where the dividing set on $T^2\times \{t\}$ is now $\Gamma_0$. Let $\mathcal{W}_0$ denote $(W,Z,[\xi_0])$. Note that $\mathcal{W}_0$ is still a special cobordism, now from $(-S^3(K_0),-\Gamma_0)$ to $(-S^3(K_1),-\Gamma_0)$. 

Consider the bypass attachment curve $\gamma_-$ defined in the previous section. If we have a special cobordism $\mathcal{W}_i = (W,Z,[\xi_i])$ from $(-S^3(K_0),-\Gamma_i)$ to $(-S^3(K_1),-\Gamma_i)$ we may attach a thickened bypass along $\gamma_- \times I$ to obtain $\mathcal{W}_{i+1} = (W,Z,[\xi_{i+1}])$, a special cobordism from $(-S^3(K_0),-\Gamma_{i+1})$ to $(-S^3(K_1),-\Gamma_{i+1})$. Each $\mathcal{W}_i$ induces a map on sutured Floer homology; we obtain a sequence of maps \[g _i : SFH(-S^3(K_0),-\Gamma_i) \to SFH(-S^3(K_1),-\Gamma_i).\] 
We claim that the following diagrams commute for each integer $i\ge 0$:
\[
\xymatrix{
SFH(-S^3(K_0),-\Gamma_i) \ar[d]^{\phi _-} \ar[r]^{g_i} &SFH(-S^3(K_1),-\Gamma_i)\ar[d]^{\phi_-}\\
SFH(-S^3(K_0),-\Gamma_{i+1}) \ar[r]^{g_{i+1}} &SFH(-S^3(K_1),-\Gamma_{i+1})}
\]
and
\[
\xymatrix{
SFH(-S^3(K_0),-\Gamma_i) \ar[d]^{\phi _+} \ar[r]^{g_i} &SFH(-S^3(K_1),-\Gamma_i)\ar[d]^{\phi_+}\\
SFH(-S^3(K_0),-\Gamma_{i+1}) \ar[r]^{g_{i+1}} &SFH(-S^3(K_1),-\Gamma_{i+1})}
\]
where the maps $\phi_-$ and $\phi_+$ were defined in the previous section.

Recall that the maps $\phi_-$ and $\phi_+$ are contact gluing maps, hence can be thought of as being induced by boundary cobordisms of sutured manifolds. The maps $g_i$ are induced by special cobordisms of sutured manifolds. Each composition of maps in the diagrams is induced by a cobordism which is obtained by stacking special and boundary cobordisms. In light of Remark \ref{order}, the diagrams commute.

The commutativity of the first diagram shows that the maps $\{g_i\}$ induce a map \[\mathcal{G}_\mathcal{W}  :\underrightarrow{SFH}(-S^3,K_0)\to\underrightarrow{SFH}(-S^3,K_1).\] The commutativity of the second diagram is equivalent to  $\mathcal{G}_\mathcal{W} $ being an $\mathbb{F}[U]$-module homomorphism.
\end{proof}

We now show that the maps $\mathcal{G}_\mathcal{W}$ are well behaved with respect to the maps $\mathcal{F}_\mathcal{W}$ defined by Juh\'{a}sz.

\begin{proposition}\label{behave}
Let $(F,\sigma)$ from $(K_0,P_0)$ to $(K_1,P_1)$ be a decorated concordance. We have the following commutative diagram:
\[
\xymatrix{
\underrightarrow{SFH}(-S^3,K_0) \ar[d]^{\Phi_{SV}} \ar[r]^{\mathcal{G}_\mathcal{W}} &\underrightarrow{SFH}(-S^3,K_1)\ar[d]^{\Phi_{SV}}\\
SFH(-S^3(K_0),-\Gamma_\mu) \ar[r]^{\mathcal{F}_\mathcal{W}} &SFH(-S^3(K_1),-\Gamma_\mu)}
\]
where the maps $\Phi_{SV}$ are those which appear in Theorem \ref{SVC}.
\end{proposition}
\begin{proof}
For any $i\ge 0$ we may attach a Stipsicz V\'{e}rtesi bypass to go from $(-S^3(K),-\Gamma_i)$ to $(-S^3(K),-\Gamma_\mu)$. Thickening this bypass and attaching it to $\mathcal{W}_i$ recovers $\mathcal{W}$. Using Remark \ref{order} again, we see that the following commutes for each $i\ge 0$:
\[
\xymatrix{
SFH(-S^3(K_0),-\Gamma_i) \ar[d]^{\phi_{SV}} \ar[r]^{g_i} &SFH(-S^3(K_1),-\Gamma_i)\ar[d]^{\phi_{SV}}\\
SFH(-S^3(K_0),-\Gamma_\mu) \ar[r]^{\mathcal{F}_\mathcal{W}} &SFH(-S^3(K_1),-\Gamma_\mu)}
\]
The result is now clear, since the vertical and top horizontal maps in the desired diagram are induced by those above.
\end{proof}

\begin{remark}
In light of Theorem \ref{SVC}, the commutative diagram in the statement above is equivalent to
\[
\xymatrix{
HFK^{-}(-S^3,K_0) \ar[d]^{p_*} \ar[r]^{G_\mathcal{C}} &HFK^{-}(-S^3,K_1)\ar[d]^{p_*}\\
\widehat{HFK}(-S^3,K_0) \ar[r]^{F_\mathcal{C}} &\widehat{HFK}(-S^3,K_1)}
\]
where $F_\mathcal{C}$ is the map from Remark \ref{orientation} and the maps $p_{*}$ are induced by settings $U=0$ on the chain level.
\end{remark}

\begin{proposition}
Let $(F,\sigma)$ from $(K_0,P_0)$ to $(K_1,P_1)$ be a decorated concordance. We have the following commutative diagram:
\[
\xymatrix{
\underrightarrow{SFH}(-S^3,K_0)\ar[dr]_{\Phi_C}\ar[rr]^{\mathcal{G}_\mathcal{W}} &  & \underrightarrow{SFH}(-S^3,K_1)\ar[dl]^{\Phi_C}\\
 & SFH(-(S^3\setminus B^3),-\Gamma) & 
}
\]
where $\Gamma$ is a single suture and the map $\Phi_C$ is the one in Theorem \ref{C2H}.

\end{proposition}
\begin{proof}
Recall that for $i\ge 0$ we may attach a contact 2-handle to $(-S^3(K),-\Gamma_i)$ to get $(-(S^3\setminus B^3),-\Gamma)$. Attaching a thickened contact 2-handle to $\mathcal{W}_i$ we obtain $\mathcal{W}'$, a sutured cobordism from $(-(S^3\setminus B^3),-\Gamma)$ to a diffeomorphic copy of itself. Remark \ref{order} shows that the following diagram commutes
\[
\xymatrix{
SFH(-S^3(K_0),-\Gamma_i) \ar[d]^{\phi_{SV}} \ar[r]^{g_i} &SFH(-S^3(K_1),-\Gamma_i)\ar[d]^{\phi_{SV}}\\
SFH(-(S^3\setminus B^3),-\Gamma) \ar[r]^{\mathcal{F}_\mathcal{W'}} &SFH(-(S^3\setminus B^3),-\Gamma)}
\]
Since any embedded arc with endpoints on both ends in $S^3\times I$ is isotopic to the trivial arc, one easily sees that the cobordism $\mathcal{W}'$ is diffeomorphic to the trivial cobordism. Thus by \cite{cob} the map $\mathcal{F}_\mathcal{W'}$ is the identity. The result follows.
\end{proof}

\begin{remark}\label{Z}
The diagram of the previous theorem is equivalent to
\[
\xymatrix{
HFK^{-}(-S^3,K_0)\ar[dr]_{\pi_*}\ar[rr]^{G_\mathcal{C}} &  &HFK^{-}(-S^3,K_1)\ar[dl]^{\pi_*}\\
 & \widehat{HF}(-S^3) & 
}
\]
where $\pi_*$ is the map induced by setting $U=1$ at the chain level. (see Theorem \ref{C2H})

\end{remark}

Functoriality is a trivial consequence of \cite[Theorem~11.12]{cob} stated as Theorem \ref{tqft} in this paper.
\begin{proposition}
Let $(F_0,\sigma_0)$ from $(K_0,P_0)$ to $(K_1,P_1)$ and $(F_1,\sigma_1)$ from $(K_1,P_1)$ to $(K_2,P_2)$ be decorated concordances, denoted $\mathcal{C}_0$ and $\mathcal{C}_1$ respectively such that $\sigma_0 |_{K_1} = \sigma_1 |_{K_1}$. Let $\mathcal{C}$ denote the decorated concordance obtained by stacking $\mathcal{C}_0$ and $\mathcal{C}_1$. Then the induced maps satisfy $G_\mathcal{C} = G_{\mathcal{C}_1}\circ G_{\mathcal{C}_0}$.

\end{proposition}

\section{Gradings}
\label{sec:gradings}

Juh\'{a}sz and Marengon \cite{JM2} show that a map $\mathcal{F}_\mathcal{W}:SFH(-S^3(K_0),-\Gamma_\mu) \to SFH(-S^3(K_1),-\Gamma_\mu)$ associated to a decorated concordance $\mathcal{C}$ splits over the relative $Spin^{\mathbb{C}}$ structures on the associated cobordism of sutured manifolds $\mathcal{W}$, denoted $Spin^\mathbb{C}(\mathcal{W})$. They use this to prove that the maps preserve the Alexander grading.

\begin{definition} \label{defn:spinc}
Let $\mathcal{W} = (W,Z,[\xi])$ be a cobordism of sutured manifolds from $(M_0,\Gamma_0)$ to $(M_1,\Gamma_1)$.
An almost complex structure $J$ defined on a subset of $W$ which contains $\partial Z$
is said to be \emph{admissible},
if the field of complex tangencies in $TZ|_{\partial Z}$ is admissible in~$(Z,\Gamma_0 \cup \Gamma_1)$, and
if the field of complex tangencies in $TM_i|_{\partial M_i}$ is admissible in $(M_i,\Gamma_i)$ for each $i$.

\emph{Relative $Spin^\mathbb{C}$ structures} on $\mathcal{W}$ are homology classes of pairs $(J,P)$, where
\begin{itemize}
\item $P$ is a finite number of points in the interior of $W$,
\item $J$ is an admissible almost complex structure defined on $W \setminus P$, and
\item $s_\xi = s_{\xi_J}$, where $\xi_J$ is the field of complex tangencies along $Z$.
\end{itemize}
 If there exists a compact $1$-manifold $C \subset W \setminus \partial Z$ such that
$P$, $P' \subset C$, $(J,P)$ and $(J',P')$ are said to be \emph{homologous};
note that in this case $J|_{W \setminus C}$ and $J'|_{W \setminus C}$ are isotopic through admissible almost complex structures.
\end{definition}

\begin{proposition} (see \cite[Proposition 3.10]{JM1})
If $\mathcal{C}$ is a decorated concordance between two knots $(K_0,P_0)$
and $(K_1, P_1)$, then the induced map satisfies
\[
F_\mathcal{C} \left( \widehat{HFK}(-S^3, K_0, -P_0, i) \right) \leq \widehat{HFK}(-S^3, K_1, -P_1, i)
\]
for every $i \in \Z$. I.e. the induced map preserved the Alexander grading.
\end{proposition}

In this section we prove that the induced maps on $HFK^{-}$ also preserve the Alexander grading.

\begin{proposition}\label{palex}
If $\mathcal{C}$ is a decorated concordance between two knots $(K_0,P_0)$
and $(K_1, P_1)$, then the induced map defined in the previous section preserves the Alexander grading, that is
\[
G_\mathcal{C} \left( HFK^{-}(-S^3, K_0, i) \right) \leq HFK^{-}(-S^3,K_1, i)
\]
for every $i \in \Z$.
\end{proposition}

Recall that the bypass attachment maps are homogeneous with respect to the Alexander grading (see Remark \ref{alex}). By the commutative diagrams in the proof of Proposition \ref{main}, it suffices to show that the map 
\[g_0 : SFH(-S^3(K_0),-\Gamma_0)\to SFH(-S^3(K_1),-\Gamma_0)\]
 preserves the Alexander grading. For any $x\in SFH(-S^3(K_0),-\Gamma_0)$, we need to show that 
\[
\frac{1}{2} \langle c_1(s(x),t_0),[S_0,\partial S_0]\rangle = \frac{1}{2} \langle c_1(s(g_0(x)),t_0),[S_1,\partial S_1]\rangle
\]
where $S_i$ is a Seifert surface for $K_i$ and $t_0$ is a nonzero section from the discussion preceding \ref{EX}.\\

\begin{remark}
We may identify the boundaries of the two knot complements, since they have sutures of identical slopes. This is why we use a single section $t_0$ for both sides of the equation.
\end{remark}

We will prove two lemmas and the proposition will follow.

The following are slight modifications of \cite[Lemma 3.8 and Lemma 3.9]{JM2} suited to our purposes:
\begin{lemma} 
Let $\mathcal{C} = (F,\sigma)$ be a decorated concordance between $(K_0,P_0)$ and $(K_1,P_1)$. Let $\mathcal{W}_0 = (W,Z,[\xi_0])$ be the induced cobordism of sutured manifolds, between $(-S^3(K_0),-\Gamma_0)$ and $(-S^3(K_1),-\Gamma_0)$, constructed in section \ref{sec:concmapshfk-}. Then the map induced on sutured Floer homology by $\mathcal{W}_0$ splits over the relative $Spin^\mathbb{C}$ structures.
\[
g_0 = \bigoplus\limits_{s\in Spin^\mathbb{C}(\mathcal{W}_0)}g_{0,s}
\]
\end{lemma}
Furthermore, the restriction maps $r_i : Spin^\mathbb{C} (\mathcal{W}_0)\to Spin^\mathbb{C}(-S^3(K_i),-\Gamma_0)$ are isomorphisms for $i\in \{0,1\}$, and $Spin^{\mathbb{C}}(\mathcal{W}_0)$ is an affine space over $H^2 (W,Z) \approx \mathbb{Z}$.
\begin{proof}
The exact argument used in proving \cite[Lemma 3.8]{JM2} applies here.
\end{proof}
 
\begin{lemma}
 Let $\mathcal{C} = (F,\sigma)$ be a decorated concordance from $(K_0,P_0)$ to $(K_1,P_1)$. Let $\mathcal{W}_0$ be as in the above lemma. For $i\in \{0,1\}$, let $S_i$ be a Seifert surface for $K_i$. Let $t_0$ be a section of $v_0^\perp$, a plane field which induces the standard characteristic foliations on $\partial (-S^3(K_i),-\Gamma_0)$ for $i\in \{0,1\}$. Then for any $s\in Spin^\mathbb{C} (\mathcal{W}_0)$ we have
 \[
\langle c_1(r_0(s),t_0),[S_0,\partial S_0]\rangle = \langle c_1(r_1 (s),t_0),[S_1,\partial S_1]\rangle
 \]
 where the maps $r_i : Spin^\mathbb{C} (\mathcal{W}_0)\to Spin^\mathbb{C}(-S^3(K_i),-\Gamma_0)$ for $i\in \{0,1\}$ are the restriction maps.
\end{lemma}

\begin{proof}
$\mathcal{W}_0 = (W,Z,[\xi_0])$. Let $d:Z\to S^1 \times S^1 \times I$ be a diffeomorphism which maps $\xi_0$ to an $I$ invariant contact structure to which we refer to by the same name(see the proof of Proposition \ref{main}). The first $S^1$ factor is identified with the fiber direction of the unit normal bundle of $F$, whose total space is $Z$. In this contact structure, $S^1\times S^1\times \{a\}$ is convex for each $a\in I$ and has dividing set $Q_0\times S^1 \times \{a\}$ (two oppositely oriented longitudes), where $Q_0\subset S^1$ consists of two oppositely oriented points. $S^1\times \{\theta\}\times I$ is convex for each $\theta \in S^1$ and has dividing set $Q_0\times \{\theta\}\times I$. Thus we can perturb the plane field $\xi_0$ to be always tangent to the first factor, i.e. the meridional direction. Indeed, we can make it so that $\xi_0$ induces the standard characteristic foliation on each $S^1\times S^1\times \{a\}$ compatible with the $-\Gamma_0$ dividing curves.

Let $J$ be an admissible almost complex structure on $W\setminus P$ which represents $s\in Spin^\mathbb{C}(\mathcal{W}_0)$, where $P\subset int(W)$ is a finite collection of points. Let $\xi_J$ denote the field of almost complex tangencies of $J$ over $Z$(now identified with $S^1\times S^1\times I$). By definition we have that $s_{\xi_0} = s_{\xi_J}$. Let $v$ be a nowhere zero section of $\xi_J$ tangent to the meridional direction. Note that both
\[
v|_{\partial (-S^3(K_0),-\Gamma_0)} \text{ and } v|_{\partial (-S^3(K_1),-\Gamma_0)}
\]
both represent the nonzero section $t_0$. The section $v$ gives rise to a trivialization $\tau$ of $\xi_J$, a complex 1 dimensional subbundle of $TW|_Z$. The complement of $\xi_J$ is trivialized by its intersection with $TZ$, hence we have a trivialization of $TW|_Z$. Since $J$ is defined over the 3-skeleton of $W$, the relative Chern class $c_1 (TW,J,t_0)\in H^2(W,Z)$ is well defined. Let $\xi_J^i$ denote the field of almost complex tangencies over $(-S^3(K_i),-\Gamma_0)$. The intersection of $\xi_J^i$ with $T(-S^3(K_i))$ gives rise to a trivialization of the complement of $\xi_J$, hence
\[
c_1(\xi_J^i,t_0) = c_1 (TW|_{(-S^3(K_i),-\Gamma_0)},J,\tau) = c_1 (TW,J,\tau)|_{(-S^3(K_i),-\Gamma_0)}
\]
Note that $\xi_J^i$ represents $r_i (s)$ by construction.

Let $S_i$ be a Seifert surface for $K_i$. Note that there is a bilinear intersection pairing
\[
H_2 (W,Z) \otimes H_2(W, (S^3(K_0)\cup -S^3(K_1)) \to \mathbb{Z}
\]
and that $H_2 (W,Z) \cong \mathbb{Z}$.
Consider the cycle $m = S^1\times \{pt\} \times I \in C_2(W, (S^3(K_0)\cup -S^3(K_1)))$. Each of the $S_i$ intersect $m$ once positively, thus they both represent the positive generator of $H_2 (W,Z)\cong \mathbb{Z}$. Hence we have:
\[
\langle c_1(r_0(s),t_0),[S_0,\partial S_0]\rangle = \langle c_1(TW,J,\tau),[S_0,\partial S_0]\rangle \]
\[= \langle c_1(TW,J,\tau),[S_1,\partial S_1]\rangle 
= \langle c_1(r_1(s),t_0),[S_1,\partial S_1]\rangle
\]
\end{proof}

Juh\'{a}sz and Marengon \cite[Section 6]{JM2} show that the map
\[F_\mathcal{C} : \widehat{HFK}(-S^{3},-K_0)\to \widehat{HFK}(-S^3,-K_1)\]
induced by a decorated concordance $\mathcal{C}$,  preserves the absolute $\mathbb{Q}$-Maslov grading. We show the following:

\begin{proposition}\label{pmaslov}
If $\mathcal{C}$ is a decorated concordance between two knots $(K_0,P_0)$
and $(K_1, P_1)$, then the induced map 
\[G_\mathcal{C}: HFK^{-}(-S^3,K_0)\to HFK^{-}(-S^3,K_1)\]
defined in the previous section preserves the absolute  $\mathbb{Z}_2$-Maslov grading.
\end{proposition}
\begin{proof}
We first show that $G_\mathcal{C}$ preserves the relative $\mathbb{Z}_2$ grading. It is enough to show that the map $g_0:SFH(-S^3(K_0),-\Gamma_0)\to SFH(-S^3(K_1),-\Gamma_0)$ preserves the relative $\mathbb{Z}_2$ grading (see Remark \ref{maslov}). Since $g_0$ is induced by a special cobordism $\mathcal{W}_0$, the methods of \cite[Section 6]{JM2} trivially carry over and show that the relative $\mathbb{Z}_2$ grading is preserved. Studying the original sutured cobordism $\mathcal{W}$, they show that 1,2, and 3 handle attachment maps (see Lemmas 6.4, 6.5, and 6.6) preserve the relative $\mathbb{Q}$ grading on $\widehat{HFK}$.

We are left to show that $G_\mathcal{C}$ preserves the absolute $\mathbb{Z}_2$ grading of some element. Consider the commutative diagram from Remark \ref{Z}
\[
\xymatrix{
HFK^{-}(-S^3,K_0)\ar[dr]_{\pi_*}\ar[rr]^{G_\mathcal{C}} &  &HFK^{-}(-S^3,K_1)\ar[dl]^{\pi_*}\\
 & \widehat{HF}(-S^3) & 
}
\]
Let $x\in HFK^{-}(-S^3,K_0)$ be an element such that $\pi_* (x)$ is a generator of $\widehat{HF}(-S^3)$. Let $gr(x)$ denote the absolute $\mathbb{Z}_2$ grading of $x$. Then we have
\[
gr(x) = gr(\pi_* (x)) = gr(\pi_*(G_\mathcal{C} (x))) = gr(G_\mathcal{C} (x))
\]
since the $U$ action has no effect on the $\mathbb{Z}_2$ grading.

\begin{conjecture*}
The map $G_\mathcal{C}:HFK^{-}(-S^3,K_0)\to HFK^{-}(-S^3,K_1)$ associated to a decorated concordance $\mathcal{C}$ preserves the absolute $\mathbb{Q}$-Maslov grading.
\end{conjecture*}

\end{proof}

\newcommand{\noop}[1]{}
\providecommand{\bysame}{\leavevmode\hbox to3em{\hrulefill}\thinspace}
\providecommand{\MR}{\relax\ifhmode\unskip\space\fi MR }
\providecommand{\MRhref}[2]{%
  \href{http://www.ams.org/mathscinet-getitem?mr=#1}{#2}
}
\providecommand{\href}[2]{#2}

\end{document}